\documentclass[12pt,a4paper,reqno]{amsart}
\usepackage[a4paper,left=16mm,right=16mm,top=36mm,bottom=36mm]{geometry}
\usepackage{bm}
 
\newtheorem{definition}{Definition}[section]
\newtheorem{remark}{Remark}[section]    
\newtheorem{theorem}{Theorem}[section]    
   
\newtheorem{lemma}{Lemma}[section]	
\newtheorem{corollary}{Corollary}[section]  
\newtheorem{proposition}{Proposition}[section]  
\newtheorem{assumption}{Assumption}

\newcommand{\R}{\mathbb{R}}

\newcommand{\T}{\mathbb{T}}

\newcommand{\vertiii}[1]{{\left\vert\kern-0.25ex\left\vert\kern-0.25ex\left\vert #1
    \right\vert\kern-0.25ex\right\vert\kern-0.25ex\right\vert}}

\def\XXint#1#2#3{{\setbox0=\hbox{$#1{#2#3}{\int}$}
     \vcenter{\hbox{$#2#3$}}\kern-.5\wd0}}

\numberwithin{equation}{section}

\usepackage[]{hyperref}
\usepackage[noabbrev, capitalise, nameinlink]{cleveref}
\crefname{equation}{}{}
\crefname{assumption}{Assumption}{Assumptions}
\crefname{rmrk}{Remark}{Remarks}
\crefformat{equation}{\textup{#2(#1)#3}}

\usepackage[colorinlistoftodos,prependcaption]{todonotes}

\begin{document}

\title[Numerical approximation of effective diffusivities]{Numerical approximation of effective diffusivities in homogenization of nondivergence-form equations with large drift by a Lagrangian method}

\author[T. Sprekeler]{Timo Sprekeler}
\address{Department of Mathematics, Texas A{\&}M University, College Station, TX 77843, USA.}
\email{timo.sprekeler@tamu.edu}

\author[H. Wu]{Han Wu}
\address{The University of Hong Kong, Department of Mathematics, Pokfulam Road, Hong Kong, China.}
\email{u3009967@connect.hku.hk}

\author[Z. Zhang]{Zhiwen Zhang}
\address{The University of Hong Kong, Department of Mathematics, Pokfulam Road, Hong Kong, China.}
\email{zhangzw@hku.hk}

\date{\today}

\begin{abstract}
    In this paper, we study numerical methods for the homogenization of linear second-order elliptic equations in nondivergence-form with periodic diffusion coefficients and large drift terms. Upon noting that the effective diffusion matrix can be characterized through the long-time variance of an associated diffusion process, we construct a Lagrangian numerical scheme based on a direct simulation of the underlying stochastic differential equation and utilizing the framework of modified equations, thereby avoiding the need to solve the Fokker--Planck--Kolmogorov equation. Through modified equation analysis, we derive higher-order weak convergence rates for our method. Finally, we conduct numerical experiments to demonstrate the accuracy of the proposed method. The results show that the method efficiently computes effective diffusivities, even in high dimensions.
\end{abstract}

\maketitle

\section{Introduction}

We study the numerical approximation of the effective diffusion matrix arising in the periodic homogenization of linear nondivergence-form second-order elliptic equations with large drift terms. Specifically, we consider the problem
\begin{equation}\label{Leps}
\begin{aligned}
    \mathcal{L}^\varepsilon u^{\varepsilon} := -A\left(\frac{x}{\varepsilon}\right):\nabla^2 u^{\varepsilon} - \frac{1}{\varepsilon}\,b\left(\frac{x}{\varepsilon}\right)\cdot \nabla u^\varepsilon &= f\quad \text{in }\Omega,\\
    u^\varepsilon &= g\quad \text{on }\partial \Omega,
    \end{aligned}
\end{equation} 
where $\Omega \subset \mathbb{R}^d$ is a sufficiently regular bounded domain, $\varepsilon > 0$ is a small scaling parameter, $A \in C^{\infty}(\mathbb{T}^d ; \mathbb{R}_{\mathrm{sym}}^{d\times d})$ is a uniformly elliptic diffusion matrix satisfying $\lambda I_d \le A \le \Lambda I_d$ in $\mathbb{T}^d$ for some $0<\lambda\leq \Lambda$, $b \in C^\infty(\mathbb{T}^d ; \mathbb{R}^d)$ is a drift coefficient, and $f:\Omega\rightarrow\mathbb{R}$ and $g:\partial\Omega\rightarrow \mathbb{R}$ are sufficiently regular. Throughout this work, $\mathbb{T}^d := \mathbb{R}^d/\mathbb{Z}^d$ denotes the $d$-dimensional flat torus, and $I_d$ the $d\times d$ identity matrix. 

The differential operator $\mathcal{L}^\varepsilon$ in \eqref{Leps} is the infinitesimal generator for the stochastic differential equation
\begin{equation}\label{diffusion with large drift}
\mathrm{d} X^{\varepsilon}_t=\frac{1}{\varepsilon}\, b\left(\frac{X^{\varepsilon}_t}{\varepsilon}\right) \mathrm{d} t+\sigma\left(\frac{X^{\varepsilon}_t}{\varepsilon}\right)  \mathrm{d} W_t,\qquad X_0^\varepsilon = x,
\end{equation}
where $X^\varepsilon_t$ is a $d$-dimensional stochastic processes, $W_t$ is a standard $d$-dimensional Wiener process (Brownian motion), $\sigma:=\sqrt{2} A^{\frac{1}{2}}$ is the diffusion coefficient matrix, and the initial condition $x\in \mathbb{R}^d$ is deterministic.

By a scaling argument, we consider the rescaled process 
\begin{equation}\label{rescaled process}
X_t := \frac{1}{\varepsilon}X^\varepsilon_{\varepsilon^2 t},
\end{equation}
which preserves the probability law of the original diffusion. The rescaled stochastic process $X_t$ is the solution to the problem
\begin{equation}\label{rescaled SDE}
 \mathrm{d}X_t =  b(X_t) \,\mathrm{d}t +\sigma(X_t)\, \mathrm{d} W_t,\qquad X_0 = \frac{x}{\varepsilon}.
\end{equation} 
The infinitesimal generator of this process is given by
\begin{equation}\label{L}
\mathcal{L}v := -A:\nabla^2 v -b\cdot \nabla v.
\end{equation}
 
For the rescaled SDE \eqref{rescaled SDE}, there exists a unique invariant probability measure $\pi$ on the torus $\mathbb{T}^d$, absolutely continuous with respect to the Lebesgue measure with a density $r(y) = \frac{\mathrm{d}\pi}{\mathrm{d}y}$. This density satisfies the following properties. For every $\phi \in C^{\infty}(\mathbb{T}^d)$, there exist constants $K, \rho > 0$ such that 
\begin{equation}\label{Invariant measure}
	\sup_{x \in \mathbb{T}^d} \left|\, \mathbb{E}[\phi(X_t) \mid X_0 = x] - \int_{\mathbb{T}^d} \phi(y) r(y)\, \mathrm{d}y \,\right| \leq K e^{-\rho t}.
\end{equation}
Moreover, the density $r$ is the unique solution to the stationary Fokker--Planck--Kolmogorov equation 
\begin{equation*}
	\mathcal{L}^* r := -\nabla^2 : (A r) + \nabla \cdot (b r) = 0 \quad \text{in } \mathbb{T}^d,
\end{equation*}
subject to the normalization condition $\int_{\mathbb{T}^d} r(y)\,\mathrm{d}y = 1$. 

The existence, uniqueness, and regularity of $r$ follow from classical results in ergodic theory and PDEs (see, e.g.,  \cite{BLP11, PS08}). In particular, the connection between homogenization theory for \eqref{Leps} and SDEs has been rigorously studied in \cite[Chapter 3]{BLP11} using probabilistic methods.
 
We introduce the corrector $\chi = (\chi_1,\ldots,\chi_d):\mathbb{T}^d\rightarrow \R^d$ as the unique solution to the cell problem
\begin{equation}\label{cell_problem_b_bar}
    \mathcal{L}\chi = b \quad\text{in }\mathbb{T}^d, \qquad \int_{\mathbb{T}^d}\chi(y)\, \mathrm{d}y = 0,
\end{equation} 
where $\mathcal{L}$ is the infinitesimal generator of the rescaled SDE defined in \eqref{L}. The existence and uniqueness of $\chi $ are guaranteed by the Fredholm alternative, provided the following centering condition holds:
\begin{align}\label{centering}
	\bar{b} := \int_{\mathbb{T}^d} b(y) r(y)\, \mathrm{d}y = 0,
\end{align}
where $r$ is the invariant density from the stationary Fokker--Planck--Kolmogorov equation. Assuming \eqref{centering}, it is a classical result in homogenization theory that the solution $u^{\varepsilon}$ of \eqref{Leps} converges weakly in $H^1(\Omega)$ to the solution $u$ of the homogenized problem
\begin{equation*}
\begin{aligned}
    \overline{\mathcal{L}} \overline{u}: =-\overline{A}:\nabla^2 \overline{u}  &= f\quad \text{in }\Omega,\\
    \overline{u} &= g\quad \text{on }\partial \Omega,
    \end{aligned}
\end{equation*} 
where the effective diffusion matrix $\overline{A}$ is a constant, symmetric and positive definite matrix given by
\begin{align}\label{Effective diffusivity}
\overline{A} := \int_{\T^d} \left(I_d + \nabla \chi(y)\right)A(y)\left(I_d+\nabla \chi(y)\right)^{\mathrm{T}}r(y)\,\mathrm{d}y; 
\end{align} 
see, e.g., \cite{BLP11,ES08,JKO94,JZ23}. This representation of $\overline{A}$ in terms of the solution to the cell problem \eqref{cell_problem_b_bar} is called the Eulerian framework for the effective diffusivity. 

We refer to \cite{CSS20,GST22,GTY20,KL16,Spr24,ST21} for optimal rates in periodic homogenization, and to \cite{AFL22,AL17,AS14,GT23,GT22} for stochastic homogenization of nondivergence-form equations. Regarding the numerical homogenization of nondivergence-form equations, recent work has led to significant developments with \cite{CSS20,FGP24,Spr24} addressing linear problems and \cite{GSS21,KS22,QST24} focusing on Hamilton--Jacobi--Bellman equations. For divergence-form equations with large drift, we refer to \cite{BLL24,BFP24,HO10,LLM17,LPS18,ZC23}.

In this paper, our objective is to develop and rigorously analyze a probabilistic numerical method to approximate the density of the invariant measure $r$ and the effective diffusion matrix $\overline{A}$. While finite element methods for approximating the effective diffusion matrix $\overline{A}$ have been rigorously analyzed in \cite{Spr24,SSZ25}, their computational cost becomes
expensive in high-dimensional settings. 

The homogenization theory developed in \cite{BLP11,PS08} establishes  that the stochastic process $X^\varepsilon_t$ defined in \eqref{diffusion with large drift}
converges in distribution to a Brownian motion with covariance matrix $\sqrt{2}\,\overline{A}^{\frac{1}{2}}$, that is, $X^\varepsilon_t \overset{d}{\to}\sqrt{2}\,\overline{A}^{\frac{1}{2}}W_t$ as $\varepsilon \to 0$, where $W_t$ is a standard $d$-dimensional Wiener process.  See also \cite{FP94,Gar97,HP08,LOY98,PS08} for homogenization of convection-diffusion equations via probabilistic methods. This probabilistic perspective motivates our particle-based approach, which exploits the Lagrangian representation of the effective diffusivity, i.e.,
\begin{equation*}
    \overline{A}^L := \lim\limits_{T \to \infty}\frac{\mathrm{Var}(X_T)}{2T},
\end{equation*}
where the covariance tensor is defined as $\mathrm{Var}(X_T) := \mathbb{E}\left[\left(X_T - \mathbb{E}[X_T]\right)\otimes\left(X_T - \mathbb{E}[X_T]\right)\right]$ and $X_t$ is the 
diffusion process from \eqref{rescaled process}. This formulation naturally circumvents the curse of dimensionality through Monte Carlo sampling.

The Monte Carlo approach for computing the effective diffusivity has been successfully employed in various settings. In particular, \cite{PSB06} developed a Monte Carlo method to approximate the effective diffusivity of inertial particles. For divergence-free velocity fields, \cite{PSZ09} implemented a stochastic splitting method that demonstrated superior performance compared to standard Euler--Maruyama discretizations. This advantage was further corroborated by \cite{WXZ18}, in which a structure-preserving symplectic scheme was introduced. Both studies found that the Euler--Maruyama scheme fails to accurately capture effective diffusivities, especially in the vanishing molecular diffusivity regime, i.e., as the diffusion coefficient tends to zero.

We will establish a rigorous connection between the Lagrangian and Eulerian formulations for computing the effective diffusivity. By the It\^{o} formula, it is proved that a universal diffusion process with sufficiently regular drift and diffusion terms still preserves the cell problem
\begin{align*}
    \mathcal{L}\chi' = b - \bar{b}\quad\text{in }\mathbb{T}^d, \qquad \int_{\mathbb{T}^d}\chi'(y)\, \mathrm{d}y = 0,
\end{align*}
and the Eulerian form limit
 \begin{equation*}
    \lim_{T\to \infty}\frac{\mathrm{Var}(X_T)}{2T}  =\int_{\T^d} \left(I_d + \nabla \chi'(y)\right)A(y)\left(I_d+\nabla \chi'(y)\right)^{\mathrm{T}}r(y)\,\mathrm{d}y,
\end{equation*}
where centering $\bar{b}=\int_{\mathbb{T}^d}b(y)r(y)\,\mathrm{d}y = 0$ may not hold so that the standard homogenization theorem does not apply. Our analysis also clarifies the limitations of the Euler--Maruyama method observed in \cite{PSZ09,WXZ18,WXZ21}.

To overcome these limitations, we use a Milstein scheme with a modified drift $b_h$ and a modified diffusion $\sigma_h$, i.e.,
\begin{equation*}
	\Tilde{X}_{t_{n+1}}^h = \Tilde{X}_{t_{n}}^h + b_{h}\left(\Tilde{X}_{t_{n}}^h\right)h + \sigma_{h}\left(\Tilde{X}_{t_n}^h\right)\sqrt{h}\,\xi_n + M_n\left(\Tilde{X}_{t_n}^h\right),\quad \tilde{X}_{0}^h = 0,\quad n= 1,2,\dotsc,N.
\end{equation*}
 The effective diffusivity is approximated via the empirical variance
 \begin{equation*}
 	\frac{\mathrm{Var}(\bm x_{t_N})}{2T} = \sum_{i=1}^M\frac{(\bm x^i_{t_N} -  \bar{\bm x}_{t_N})\otimes(\bm x^i_{t_N} - \bar{\bm x}_{t_N})}{2MT},
\end{equation*}
where $\bm x_{t_N} =  \left(\bm x_{t_N,[1]},\bm x_{t_N,[2]},\dotsc, \bm x_{t_N,[d]}\right)^T $ represents $M$ independent realizations of the numerical solution at time $T=Nh$ and $\bar{\bm x}_{t_N} = \frac{1}{M}\sum_{i=1}^M\bm x^i_{t_N}$ is the sample mean.  

Using backward error analysis, we obtain the error estimate of the invariant measure and the solution of the cell problem from the modified flow. Then, an approximation error for the effective diffusivity of order $\mathcal{O}\left(h^2 +\frac{1}{T} +\frac{1}{\sqrt{M}}\right)$ can be deduced from these estimates. Finally, we present numerical experiments to demonstrate the accuracy of the proposed method. The results show that our method efficiently computes effective diffusivities, even in three-dimensional settings. 
 
This paper is organized as follows. Section 2 establishes the theoretical foundations of the Lagrangian framework for homogenization, proving the equivalence between pathwise and PDE-based formulations of effective diffusivity. Building on this framework, Sections 3 and 4 develop our probabilistic numerical method featuring a modified high-order discretization scheme. Further, Section 4 provides a complete convergence analysis, deriving rigorous error estimates for both the approximation of the invariant measure and the computation of the effective diffusivity in high dimensions. Numerical validation in Section 5 demonstrates the accuracy of the method through benchmark tests and applications to convection-enhanced diffusion in chaotic flows.

\section{Theoretical Foundations of the Lagrangian Method}

We begin by analyzing the relationship between the Lagrangian and Eulerian frameworks for computing the effective diffusivity. The homogenization theory developed in \cite{BLP11,JZ23,PS08} provides a rigorous analysis of diffusion processes with large drift, while \cite{BLP11} establishes the probabilistic connection between the elliptic problem \eqref{Leps} and stochastic differential equations.

The long-time behavior of diffusion processes in periodic potentials has been studied in \cite{Cap98,Gar97}, though these results require the centering condition \eqref{centering} to hold. When implementing numerical integrators, the modified SDE obtained through backward error analysis introduces both a perturbed invariant measure and modified drift terms. This modification typically violates the centering condition, thereby preventing the standard homogenization result from being achieved. This phenomenon provides a theoretical explanation for the failure of the Euler--Maruyama scheme observed in \cite{PSZ09,WXZ18}. A detailed analysis of these effects will be presented in Section 4.

When the centering condition fails to hold, our analysis nevertheless establishes that $X_t$ converges in distribution to a Brownian motion with covariance matrix $\overline{A}^L$. 
This convergence result reveals the fundamental connection between the pathwise diffusion coefficient $\overline{A}^L$ and its Eulerian counterpart, even in non-centered systems. These theoretical findings directly motivate our computational framework, where we approximate $\overline{A}^L$ through long-time statistics of large particle ensembles, as developed in Section 3. 
 
The following theorem establishes the rigorous connection between the Lagrangian and Eulerian formulations, thereby providing the theoretical foundation for our proposed Lagrangian method. For $M\in \R^{d\times d}$, we write $\lvert M \rvert := \sqrt{M:M}$ to denote the Frobenius norm of $M$.

 \begin{theorem}\label{Relation between Lagrangian and Eulerian frameworks}
 Consider the diffusion process $X_t$ governed by the SDE 
    \begin{equation}\label{target SDE}
        \mathrm{d}X_t = b(X_t)\,\mathrm{d}t + \sigma(X_t)\,\mathrm{d}W_t,\qquad X_0 = x_0,
\end{equation}
where the drift $b:\mathbb{T}^d\rightarrow \mathbb{R}^d$ and diffusion coefficient $\sigma:\mathbb{T}^d\rightarrow \mathbb{R}^{d \times d}$ are assumed to be smooth and such that the diffusion matrix $A := \frac{1}{2}\sigma\sigma^{\mathrm{T}}$ is uniformly elliptic. Let $r:\mathbb{T}^d\rightarrow (0,\infty)$ denote the density function of the invariant measure satisfying \eqref{Invariant measure}. Let $\chi' :\mathbb{T}^d \rightarrow \mathbb{R}^d$ be the unique solution to the cell problem
    \begin{equation}\label{chip eqn}
        \mathcal{L}\chi' = b - \bar{b}\quad\text{in }\mathbb{T}^d, \qquad \int_{\mathbb{T}^d}\chi'(y)\,\mathrm{d}y = 0,
    \end{equation}
    where $\bar{b}:=\int_{\mathbb{T}^d}b(y)r(y)\,\mathrm{d}y$ (which may be nonzero). Then, there holds
     \begin{align*}
 		\bar{b}^L &:=\lim_{t\to \infty}\frac{\mathbb{E}[X_t|X_0 = x_0]- x_0}{t}  = \bar{b},\\ \overline{A}^L &:= \lim_{t\to \infty}\frac{\mathrm{Var}(X_t)}{2t}  =\int_{\mathbb{T}^d} \left(I_d + \nabla \chi'(y)\right)A(y)\left(I_d+\nabla \chi'(y)\right)^{\mathrm{T}}r(y)\,\mathrm{d}y.
 \end{align*} 
\end{theorem}

\begin{proof}
     Applying the It\^{o} formula to $\chi'$ and using \eqref{chip eqn}, we obtain that
    \begin{align*}
            \chi'(X_t) &= \chi'(x_0) +\int_0^t \nabla \chi'(X_s)\, \sigma (X_s)\,\mathrm{d}W_s - \int_0^t\mathcal{L}\chi'(X_s)\, \mathrm{d}s\\
            &= \chi'(x_0) +\int_0^t \nabla \chi'(X_s)\, \sigma (X_s)\,\mathrm{d}W_s - \int_0^t \left(b(X_s)-\bar{b}\right) \mathrm{d}s.
    \end{align*}  
    In view of the SDE \eqref{target SDE}, we deduce that 
    \begin{equation}\label{X_t_chi}    
        X_t = x_0 - \left(\chi'(X_t) - \chi'(x_0)\right) + t\, \bar{b} + 
 \int_0^t \left(I_d + \nabla \chi'(X_s) \right) \sigma(X_s)\,\mathrm{d} W_s.
        \end{equation}
 Taking expectations on both sides, noting that $\int_0^t \left(I_d + \nabla \chi'(X_s) \right) \sigma(X_s)\,\mathrm{d} W_s$ is a martingale whose expectation vanishes, and noting that $\chi'\in L^{\infty}(\mathbb{T}^d)$ by elliptic regularity theory and Sobolev embeddings, we obtain that
\begin{equation}\label{moment_est}
        \left|\frac{\mathbb{E}_{x_0}[X_t] -x_0}{t} - \bar{b}\right| \le  \frac{\left|\mathbb{E}_{x_0}[\chi'(X_t)] - \chi'(x_0)\right|}{t}\le  \frac{ 2\|\chi'\|_{\infty}}{t},
\end{equation}
where $\mathbb{E}_{x_0}[X_t]:=\mathbb{E}[X_t|X_0 = x_0]$. Hence, we find that
\begin{equation*}
    \lim_{t\rightarrow \infty}\left|\frac{\mathbb{E}_{x_0}[X_t] -x_0}{t} - \bar{b}\right| = 0.
\end{equation*}
Meanwhile, writing 
\begin{align*}
    \overline{A}^L := \int_{\mathbb{T}^d} \left(I_d + \nabla \chi'(y)\right)A(y)\left(I_d+\nabla \chi'(y)\right)^{\mathrm{T}}r(y)\,\mathrm{d}y,
\end{align*}
taking the variance of the term $\int_0^t \left(I_d + \nabla \chi'(X_s)\right) \sigma(X_s)\,\mathrm{d} W_s$ from \eqref{X_t_chi}, and using Fubini's theorem yields that
\begin{align*}
        \alpha_t &:=\frac{1}{2t}\,\mathrm{Var}\left[\int_0^t \left(I_d + \nabla \chi'(X_s)\right)\sigma(X_s)\,\mathrm{d}W_s\right] - \overline{A}^L \\&\;= \frac{1}{t}\,\mathbb{E}_{x_0}\left[\int_0^t \left(I_d + \nabla \chi'(X_s)\right)A(X_s)\left(I_d + \nabla \chi'(X_s)\right)^{\mathrm{T}}\,\mathrm{d}s\right] - \overline{A}^L\\
        &\;= \frac{1}{t}\int_0^t\int_{\mathbb{T}^d} \left(I_d + \nabla \chi'(y) \right)A(y) \left(I_d + \nabla \chi'(y) \right)^{\mathrm{T}}r_s(y) \,\mathrm{d}y\, \mathrm{d}s - \overline{A}^L\\
        &\;= \frac{1}{t} \int_0^t\int_{\mathbb{T}^d} \left(I_d + \nabla \chi'(y) \right)A(y) \left(I_d + \nabla \chi'(y) \right)^{\mathrm{T}}\left(r_s(y)-r(y)\right) \mathrm{d}y\, \mathrm{d}s,
\end{align*}
where $A:=\frac{1}{2}\sigma\sigma^{\mathrm{T}}$ and $r_s$ denotes the probability density of $X_s$. In view of \eqref{Invariant measure}, we deduce that there exist constants $K,\rho>0$ such that
\begin{align}\label{bd on at}
    \lvert \alpha_t \rvert \leq \frac{1}{t}\left|\int_0^t K \mathrm{e}^{-\rho s}\,\mathrm{d}s\right| \le \frac{K}{\rho t}(1 - \mathrm{e}^{-\rho t}) \leq \frac{K}{\rho t}.
\end{align}
Then, we use \eqref{X_t_chi} and \eqref{bd on at} to find that
\begin{align}\label{rate 1overt}
\begin{split}
        &\left\lvert\frac{\mathrm{Var}(X_t)}{2t} - \overline{A}^L \right\rvert \\ &\le   
  \frac{\left\lvert \mathrm{Var}\left( \chi'(X_t)\right)\right\rvert}{2t} + \frac{1}{t} \left\lvert \mathbb{E}_{x_0}\hspace{-0.15cm}\left[\left(\int_0^t\left(I_d + \nabla \chi'(X_s)\right)\sigma(X_s)\,\mathrm{d}W_s\hspace{-0.1cm}\right) \hspace{-0.1cm}\left(\chi'(X_t) - \mathbb{E}_{x_0}[\chi'(X_t)]\right)^{\mathrm{T}} \right]\right\rvert + \lvert \alpha_t \rvert \\
        &\le \frac{2\|\chi'\|^2_\infty}{t} + \frac{K}{\rho t},
        \end{split}
    \end{align}
where we used in the final step that $\left(\int_0^t\left(I_d + \nabla \chi'(X_s)\right)\sigma(X_s)\,\mathrm{d}W_s\right)\left(\chi'(X_t) - \mathbb{E}_{x_0}[\chi'(X_t)]\right)^{\mathrm{T}}$ is a martingale whose expectation vanishes.

Taking the limit as $t\rightarrow \infty$, we conclude that
\begin{align*}
    \lim_{t\rightarrow \infty} \left\lvert\frac{\mathrm{Var}(X_t)}{2t} - \overline{A}^L \right\rvert = 0,
\end{align*}
which completes the proof.
\end{proof} 

\begin{remark} 
    While existing results  \cite{HO11,JZ23,PS05,JKO94} establish that the failure of the centering condition \eqref{centering} generally precludes standard homogenization results, our analysis reveals that the rescaled SDE framework nevertheless maintains convergence to well-defined Eulerian limits. This persistence of structure in the Lagrangian formulation is useful for the development of numerical approximation schemes.
\end{remark}

\section{Development of Numerical Schemes}

In this section, we present a numerical framework for computing the effective diffusivity through stochastic simulation of the rescaled SDE \eqref{rescaled SDE}. Our approach integrates high-order discretization schemes with Monte Carlo estimation, utilizing enhanced versions of the Euler--Maruyama and Milstein methods \cite{MT21} improved by the modified equation approach \cite{ACV12} to achieve high-order accuracy in solving the underlying SDEs. 

\subsection{One-step numerical integrators}

For the numerical approximation of the diffusion process $X_t$ in \eqref{target SDE}, we implement a one-step numerical integrator defined recursively by
\begin{equation}\label{numerical integrator}
	X^h_{t_{n}} = \Phi^{b,\sigma}(X^h_{t_{n-1}}, h, \zeta_{n-1}), \quad t_n = nh \in (0,T], \quad n=1,\dots,N,
\end{equation}
where $h > 0$ denotes the fixed time-step size, $\{\zeta_n\}_{n=0}^{N-1}$ are i.i.d. $d$-dimensional random vectors, and $\Phi^{b,\sigma}: \mathbb{R}^d \times (0,\infty) \times \mathbb{R}^d \to \mathbb{R}^d$ is the flow map that incorporates both drift $b$ and diffusion $\sigma$ terms. The numerical approximations $X^h_{t_n}$ correspond to the exact solution at times $t_n  = nh\in (0,T]$, $n= 1,2,...,N$. The accuracy of the approximation can be measured in terms of the weak order of the numerical integrator defined as follows.

\begin{definition}\cite{MT21}
    A numerical integrator \eqref{numerical integrator} starting from the exact initial condition $X_0^h = X_0$ is said to have weak order $p$ for the approximation of $X_{t}$ if for all test functions $f \in C_P^{\infty}(\mathbb{R}^d)$ there holds
    \begin{equation*}
        \left|\mathbb{E}\left[f\left(X_{t_n}^h\right) - \mathbb{E}[f(X_{t_n})]|X_0 = x\right]\right| = \mathcal{O}( h^p),\qquad 0\leq t_n = nh\leq T.
    \end{equation*}    
\end{definition}

Here, we used the notation $C_P^{\infty}(\mathbb{R}^d)$ to denote the class of smooth functions that together with their partial derivatives are of polynomial growth.

Two of the most widely used approaches to approximate the stochastic process are the Euler--Maruyama scheme 
\begin{equation*}
    X_{t_{n+1}}^h = X_{t_{n}}^h +b\left(X_{t_{n}}^h\right)h + \sigma\left(X_{t_{n}}^h\right)\sqrt{h}\,\xi_n, \quad n =1,2,\dotsc,N,
\end{equation*}
and the Milstein scheme 
\begin{equation}\label{Milstein scheme}
    X_{t_{n+1}}^h = X_{t_{n}}^h +b\left(X_{t_{n}}^h\right)h + \sigma\left(X_{t_{n}}^h\right)\sqrt{h}\,\xi_n  + M_n\left(X_{t_n}^h\right), \quad n =1,2,\dotsc,N,
\end{equation}
where $\xi_n\sim \mathcal{N}(0,I_d)$ is a standard $d$-dimensional normal random vector, and $M_n(X_{t_n}^h)$ denotes the Milstein correction term. It is known that both the Euler--Maruyama and Milstein method have first-order convergence in the weak sense \cite{MT21}.

In the works \cite{ACV12,AVZ14,Zyg11}, the idea behind the modified equation is that for a numerical integrator \eqref{numerical integrator} with order $p$ for the invariant measure of the SDE \eqref{rescaled process}, we can construct a modified SDE
\begin{equation}\label{modified SDE}
    \mathrm{d}\Tilde{X}_t = b_h(\Tilde{X}_t)\, \mathrm{d}t + \sigma_h(\Tilde{X}_t) \,\mathrm{d}W_t,\quad \Tilde{X}_0 = x,
\end{equation}
where $b_h$ and $\sigma_h$ admit asymptotic expansions of the form 
\begin{equation}\label{modified coefficients}
    b_h = b + h^pb_p + \dotsc + h^{p +m-1}b_{p+m-1},\qquad \sigma_h = \sigma + h^p\sigma_p + \dotsc + h^{p+m-1}\sigma_{p+m-1}
\end{equation}
with $m\in \mathbb{N}$. When the numerical integrator \eqref{numerical integrator} is applied to this modified SDE \eqref{modified SDE}, i.e.,
\begin{equation*}
    \Tilde{X}^h_{t_n} = \Phi^{b_h,\sigma_h}\left(\Tilde{X}^h_{t_{n-1}},h,\zeta_{n-1}\right),\quad n = 1,2,\dotsc,N,
\end{equation*}
the resulting approximation achieves weak order $p+m$. For a rigorous analysis and construction of modified equations, we refer to \cite{ACV12,AVZ14,Zyg11}. By employing the methodology of modified equations, we can approximate $X_t$ with a higher convergence rate.

\subsection{Improved Milstein scheme based on the modified equation approach}

We propose an improved Milstein scheme constructed via the modified equation approach, which achieves second-order weak convergence while preserving the computational efficiency of the classical method. Specifically, the second-order numerical integrator based on the Milstein scheme takes the form 
\begin{equation}\label{Euler Maruyama with modified equation}
    \Tilde{X}_{t_{n+1}}^h = \Tilde{X}_{t_{n}}^h + b_{h}\left(\Tilde{X}_{t_{n}}^h\right)h + \sigma_{h}\left(\Tilde{X}_{t_n}^h\right)\sqrt{h}\,\xi_n + M_n\left(\Tilde{X}_{t_n}^h\right),\quad \tilde{X}_{0}^h = 0,\quad n= 1,2,\dotsc,N,
\end{equation}
where 
$\xi_n \sim \mathcal{N}(0,I_d)$ is a $d$-dimensional standard Gaussian random  vector, and the Milstein correction term $M_n\left(\Tilde{X}_{t_{n}}^h\right) = \left(M_{n,[i]}\left(\Tilde{X}_{t_{n}}^h\right)\right)_{i=1}^d \in \mathbb{R}^d$ is defined component-wise by
\begin{equation*}
    M_{n,[i]}(\Tilde{X}_{t_n}^h)  = \Xi_i(\Tilde{X}_{t_n}^h) : J_n = \sum_{j_1,j_2=1}^d \Xi_{i,[j_1,j_2]}J_n^{[j_1,j_2]},\quad i = 1,2,\dotsc,d.
\end{equation*}
Here, the coefficient matrix $\Xi_i = (\Xi_{i,[j_1,j_2]})_{1\leq j_1,j_2\leq d}$ is the $d\times d$ matrix given by
\begin{equation*}
    \Xi_{i,[j_1,j_2]} = \sum_{k=1}^d \partial_k\sigma_{[i,j_2]}\sigma^{[k,j_1]},\quad i,j_1,j_2 = 1,2,\dotsc,d,
\end{equation*}
and the double It\^{o} integral $J_n = (J_n^{[j_1,j_2]})_{1\leq j_1,j_2\leq d}$ is defined as the $d \times d$ matrix with entries  
\begin{equation*}
    J_n^{[j_1,j_2]} =  \int_{t_n}^{t_{n+1}}\left(\int_{t_n}^s \mathrm{d}W^{j_1}_t\right) \mathrm{d}W^{j_2}_s,\quad j_1,j_2 = 1,2,\dotsc,d.
\end{equation*}
The integral matrix $J_n$ is difficult to evaluate in general. However, under the commutativity condition
\begin{equation*}
	\Xi_{i,[j_1,j_2]} = \Xi_{i,[j_2,j_1]} \quad \forall i,j_1,j_2 \in \{1,\dots,d\},
\end{equation*}
the matrix $J_n$ can be approximated by
\begin{equation*}
 \frac{h}{2}\left(\xi_n^{j_1}\xi_n^{j_2} - \delta_{j_1,j_2}\right), \quad j_1,j_2 = 1,\dots,d,
\end{equation*}
where $\xi_n \sim \mathcal{N}(0,I_d)$ is a $d$-dimensional standard Gaussian random vector and $\delta_{j_1,j_2}$ is the Kronecker delta. 

For the non-commutative case, we approximate $J_n$ using a truncated multiple Fourier--Legendre series of order $q$, i.e., 
\begin{equation*}
    J_n^{q,[j_1,j_2]} = \frac{h}{2}\left(\xi^{j_1}_n\xi_n^{j_2} + \sum_{k=1}^q\frac{1}{\sqrt{4k^2 -1}}(\xi^{(k-1),j_1}_n\xi_n^{(k),j_2} - \xi^{(k),j_1}_n\xi_n^{(k-1),j_2}) - \delta_{j_1,j_2}\right),
\end{equation*}
where $\xi_n^{(k)}\sim \mathcal{N}(0,I_d)$ are i.i.d. standard normal random vectors with $\xi_n^{(0)} = \xi_n$. Without altering the weak convergence rate, we can simply take $q = 2$. See Theorem 2 in \cite{Kuz20} for more details.

The modified equation coefficients $b_{h}$ and $\sigma_{h}$ are expanded as 
\begin{equation*}
	b_h = b + h\,b_1, \quad \sigma_h = \sigma + h\,\sigma_1,
\end{equation*}
where the first-order corrections $b_1 = (b_{1,[i]})_{1\leq i\leq d}$ and $\sigma_1 = (\sigma_{1,[j_1,j_2]})_{1\leq j_1,j_2\leq d}$ are given by
\begin{equation}\label{modified equation}
    \begin{aligned}
    b_{1,[i]} &= \frac{1}{2}b\cdot \nabla b_{[i]} + \frac{1}{4} (\sigma\sigma^T):\nabla^2 b_{[i]},\qquad i=1,\dots,d,\\
    \sigma_{1,[j_1,j_2]} &= \frac{1}{2} \left((\nabla b)\sigma\right)_{[j_1,j_2]} + \frac{1}{2}b\cdot \nabla \sigma_{[j_1,j_2]} + \frac{1}{4}(\sigma\sigma^T): \nabla^2 \sigma_{[j_1,j_2]},\qquad j_1,j_2=1,\dots,d,
    \end{aligned}   
\end{equation}
This is a specific case of the $\theta$-Milstein method provided in \cite{ACV12}.
More precisely, the first-order correction terms $b_1$ and $\sigma_1$ can be expressed using Einstein summation convention as follows:
 \begin{equation*}
    \begin{aligned}
        b_{1,[i]} &= \frac{1}{2}\partial_{k}b_{[i]}b^{[k]} + \frac{1}{4}(\sigma\sigma^T)_{[k_1,k_2]}\partial^{k_1}\partial^{k_2}b_{[i]},\\
        \sigma_{1,[j_1,j_2]} &= \frac{1}{2}\sigma_{[k,j_2]}\partial^k b_{[j_1]}  +\frac{1}{2} \partial_k\sigma_{[j_1,j_2]}b^{[k]} +\frac{1}{4}(\sigma\sigma^T)_{[k_1,k_2]}\partial^{k_1}\partial^{k_2}\sigma_{[j_1,j_2]}.
\end{aligned}     
 \end{equation*} 

In Section 5, we will employ the improved Milstein scheme, constructed using the modified equation approach, to numerically approximate invariant measures and effective diffusivities for SDEs associated with nondivergence-form elliptic equations involving large drift terms.

\section{Convergence analysis}  
\noindent

In this section, we establish error estimates for the proposed numerical scheme in computing both the invariant measure and solutions to the cell problem through backward error analysis. Furthermore, we derive approximation error bounds for the effective diffusivity in the associated SDEs.

\subsection{Error analysis for the invariant measure}  

We analyze the weak convergence properties of the first modified equation, following \cite{ACV12}. Consider the backward Kolmogorov equation associated with \eqref{rescaled process}, i.e.,
\begin{equation}\label{backward Komogorov equation}
    \frac{\partial u}{\partial t} = \mathcal{L}u, \qquad u(x,0) = \phi(x),
\end{equation}
where $\phi\in C^\infty(\mathbb{T}^d)$ is an arbitrary test function. With $X_t$ denoting the solution to \eqref{rescaled SDE}, we have that the solution to \eqref{backward Komogorov equation} is given by $u(x,t) = u^{b,\sigma}(\phi,x,t)$, where
\begin{equation*}
u^{b,\sigma}(\phi,x,t) = \mathbb{E}\left[\phi(X_t)\,|\,X_0 = x\right].
\end{equation*}
Since $b,\sigma$, and $\phi$ are smooth functions, $u^{b,\sigma}$ admits the formal Taylor expansion
\begin{equation*}
    u^{b,\sigma}(\phi,x,h) = \phi(x) + \sum_{j=1}^{\infty} \frac{h^j}{j!}\mathcal{L}^j\phi(x).
\end{equation*}
For the numerical integrator \eqref{numerical integrator}, we define the numerical solution to \eqref{backward Komogorov equation} as  
\begin{equation}\label{numerical solution to backward Kolmogrov equation}
	U^{b,\sigma}(\phi,x,h) = \mathbb{E}\left[\phi(X_{t_1}^h)\,|\,X_0 = x\right].
\end{equation}
We make the following assumption about the numerical solution.
\begin{assumption}\label{Assumption 1}
    The numerical solution \eqref{numerical solution to backward Kolmogrov equation} 
    admits a weak Taylor expansion of the form
    \begin{equation*}
    U^{b,\sigma}(\phi,x,h) = \phi(x) + h\, \mathcal{L}\phi(x)+ \sum_{j=1}^{\infty}h^{j+1}\mathcal{A}_j^{b,\sigma}\phi(x)        ,
    \end{equation*}
    where $\mathcal{A}_j^{b,\sigma}$, $j \in \mathbb{N}$, are numerical operators depending on the coefficients of the SDE \eqref{rescaled SDE}.
\end{assumption}
The following result is a special case of \cite[Theorem 2.1]{ACV12}.

\begin{proposition}\label{Prop: order two}
    Consider a first weak-order numerical integrator \eqref{numerical integrator} satisfying Assumption \ref{Assumption 1}. If there exist modified coefficients $b_{h} = b + h\,b_1 \in C^{\infty}(\mathbb{T}^d;\mathbb{R}^d)$ and $\sigma_{h} = \sigma + h\,\sigma_1\in C^{\infty}(\mathbb{T}^d;\mathbb{R}^{d\times d})$ as in \eqref{modified coefficients} such that the differential operator 
    \begin{equation*}
        \mathcal{L}_{1} := \lim_{h \to 0}\frac{u^{b,\sigma}(\cdot,x,h)- U^{b_h,\sigma_h}(\cdot,x,h)}{h^2}
    \end{equation*}
    can be expressed in the form
    \begin{equation}\label{L_1 form}
        \mathcal{L}_1 =  \frac{1}{2}(\sigma \sigma_1^{\mathrm{T}} +\sigma_1 \sigma^{\mathrm{T}}):\nabla^2 + b_1\cdot \nabla,
    \end{equation}
    then the numerical integrator with the modified equation
    \begin{equation}\label{numerical integrator with modifeid equation}
        \Tilde{X}^h_{t_n+1} = \Phi_{b_h,\sigma_h}(\Tilde{X}^h_{t_n},h,\xi_n)
    \end{equation}
   achieves second-order weak convergence for the original SDE. Specifically, for any test function $f\in C_P^{\infty}(\mathbb{R}^d)$,  sufficiently small time-step $h>0$, and any fixed time $t_n =nh$, we have that
    \begin{equation*}
        \left|\mathbb{E}[f(\Tilde{X}^h_{t_n})] - \mathbb{E}[f(X_{t_n})]\right| \le C h^{2}.
    \end{equation*}
\end{proposition}

For the Milstein scheme \eqref{Milstein scheme}, one can verify that the coefficients $b_h$ and $\sigma_h$ defined in \eqref{modified equation} satisfy the operator form \eqref{L_1 form}. Consequently, the modified scheme \eqref{Euler Maruyama with modified equation} achieves second-order convergence in the weak sense.

\begin{remark}
   In the case of non-constant $\sigma$, the operator $\mathcal{L}_1$ associated with the Euler--Maruyama scheme contains third-order derivative terms. This prevents its expression in the form \eqref{L_1 form}. Therefore, a modified integrator cannot be constructed for the Euler--Maruyama method. If we use Milstein’s method instead of the Euler method, such a derivation
of a modified It\^{o} SDE is possible under appropriate smoothness assumptions for the drift $b$ and the diffusion coefficient $\sigma$.
\end{remark}

\begin{corollary}[Convergence of invariant measure]\label{Convergence of invariant measure}
    Under the assumptions of Theorem \ref{Relation between Lagrangian and Eulerian frameworks}, for any test function $\phi \in C^\infty(\mathbb{T}^d)$, there exist constants $C,K,\rho>0$, such that for $t_n = nh$, the numerical solution $\Tilde{X}_{t_n}^h$ with smooth modified coefficients $b_h,\sigma_h$ satisfies
    \begin{equation*}
        \left| \mathbb{E}[\phi(\Tilde{X}^h_{t_n})] - \int_{\mathbb{T}^d} \phi(y)r(y)\, \mathrm{d}y \right| \le K e^{-\rho t} + Ch^2.
    \end{equation*}
\end{corollary}

\subsection{Error analysis for effective diffusion matrix}

The backward error analysis approach \cite{Rei99} provides an effective framework for analyzing numerical approximations. This technique modifies the original SDE \eqref{target SDE} to construct a continuous-time process $\hat{X}^h_t$ whose SDE takes the form
\begin{align*}
    \mathrm{d}\hat{X}^h_t = \hat{b}_h\,\mathrm{d}t + \hat{\sigma}_h\,\mathrm{d}W_t,
\end{align*}
that better approximates the discrete numerical solution $X^h_{t_n}$ from \eqref{numerical integrator}. Specifically, for any smooth test function $f\in C_P^{\infty}(\mathbb{R}^d)$, sufficiently small time step $h>0$, and discrete times $t_n = nh$, we obtain the error estimate
\begin{equation}\label{BAE err}
    \left|\mathbb{E}[f(X^h_{t_n}) - \mathbb{E}f(\hat{X}^h_{t_n})]\right| \le C h^{p+q},
\end{equation}
where $p$ is the base order of the numerical method, $q$ represents the order improvement, and $C >0$ is a constant independent of $h$.

For the Milstein scheme \eqref{Milstein scheme}, the first-order modified SDE in backward error analysis \cite{DF12} takes the form 
\begin{equation*}
    \mathrm{d}\hat{X}^h_t = (b - h\,b_1)\,\mathrm{d}t + (\sigma - h\,\sigma_1 )\,\mathrm{d}W_t,
\end{equation*}
where the correction terms $b_1\in C^{\infty}(\mathbb{T}^d)$ and $\sigma_1\in C^{\infty}(\mathbb{T}^d)$ are identical to those defined in \eqref{modified equation}. This means that the modified coefficients $b_{h},\sigma_{h}$ of modified equations for backward error analysis and for modified integrators of second order are identical up to the multiplicative factor $-1$; see \cite{ACV12}.
 
 Regarding $b + h\,b_1$ and $\sigma + h\,\sigma_1$ as the new $b$ and $\sigma$, we take backward error analysis for the numerical integrator with modified equation $\Tilde{X}^h_{t_n}$ \eqref{numerical integrator with modifeid equation} of the given SDE \eqref{rescaled process} which yields
\begin{equation*}
    \mathrm{d} X'_t = (b + h^2 b'_1)\,\mathrm{d}t + (\sigma + h^2\sigma_1')\,\mathrm{d}W_t,
\end{equation*}
where the second-order correction terms $b_1' = (b_{1,[i]}')_{1\leq i\leq d}$ and $\sigma_1' = (\sigma_{1,[j_1,j_2]}')_{1\leq j_1,j_2\leq d}$ are given by
\begin{align*}
    b_{1,[i]}' &= \frac{1}{2}\Big( b_1\cdot \nabla b_{[i]}  + b\cdot \nabla b_{1,[i]} \Big)+
\frac{1}{4} \Big[ (\sigma_1 \sigma^T + \sigma \sigma_1^T):\nabla^2  b_{[i]} + (\sigma \sigma^T):\nabla^2 b_{1,[i]} \Big]
\\
& +\frac{h}{2}\, b_1\cdot \nabla b_{1,[i]}
+\frac{h}{4}\Big[ (\sigma_1 \sigma^T + \sigma \sigma_1^T) :\nabla^2 b_{1,[i]} + (\sigma_1 \sigma_1^T) :\nabla^2 b_{[i]} \Big],
\end{align*}
and
\begin{align*}
    \sigma_{1,[j_1,j_2]}' &= \frac{1}{2}\Big(\nabla b_{[j_1]}\cdot \sigma_{1,[\cdot,j_2]} + \nabla b_{1,[j_1]}\cdot \sigma_{[\cdot,j_2]}\Big) +\frac{1}{2}\Big(\nabla \sigma_{[j_1,j_2]}\cdot b_1 + \nabla \sigma_{1,[j_1,j_2]}\cdot b\Big) \\&+
    \frac{1}{4}\Big((\sigma\sigma^T):\nabla^2\sigma_{1,[j_1,j_2]}+(\sigma_1\sigma^T + \sigma\sigma_1^T):\nabla^2\sigma_{[j_1.j_2]}\Big) + \frac{h}{2}\Big(\nabla b_{1,[j_1]}\cdot \sigma_{1,[\cdot,j_2]} + \nabla \sigma_{1,[j_1,j_2]}\cdot b_{1}\Big)\\
    &+ \frac{h}{4}\Big((\sigma_1\sigma^T +\sigma\sigma_1^T):\nabla^2\sigma_{1,[j_1,j_2]} + \sigma_1\sigma_1^T:\nabla^2\sigma_{[j_1,j_2]}\Big)
\end{align*}
for $i,j_1,j_2 = 1,2,\dotsc,d$. Equivalently, using Einstein's summation convention,
\begin{equation*}
\begin{aligned}
b_{1,[i]}' &= \frac{1}{2}\Big( b_{1,[k]} \partial^{k} b_{[i]}  + \partial_{k} b_{1,[i]}\, b^{[k]} \Big)+
\frac{1}{4} \Big[ (\sigma_1 \sigma^{\mathrm{T}} + \sigma \sigma_1^{\mathrm{T}})_{[k,l]} \partial^{k} \partial^{l} b_{[i]} + (\sigma \sigma^T)_{[k,l]} \partial^{k} \partial^{l} b_{1,[i]} \Big]\\
& +\frac{h}{2}\,b_{1,[k]}\partial^{k} b_{1,[i]}
+\frac{h}{4}\Big[ (\sigma_1 \sigma^{\mathrm{T}} + \sigma \sigma_1^{\mathrm{T}})_{[k,l]} \partial^{k} \partial^{l} b_{1,[i]} + (\sigma_1 \sigma_1^{\mathrm{T}})_{[k,l]} \partial^{k} \partial^{l} b_{[i]} \Big], 
\end{aligned}
\end{equation*}
and 
\begin{equation*}
    \begin{aligned}
&\sigma_{1,[j_1,j_2]}' = \frac{1}{2} \Big( \partial_k b_{[j_1]}\sigma_1^{[k,j_2]} + \partial_k b_{1,[j_1]}\sigma^{[k,j_2]} \Big)
+\frac{1}{2} \Big( \partial_k \sigma_{[j_1,j_2]} b_1^{[k]} + \partial_k \sigma_{1,[j_1,j_2]} b^{[k]} \Big)\\
&+\frac{1}{4}  \Big[ (\sigma \sigma^{\mathrm{T}})_{[k_1,k_2]} \partial^{k_1}\partial^{k_2} \sigma_{1,[j_1,j_2]}+
(\sigma_1 \sigma^{\mathrm{T}} + \sigma \sigma_1^{\mathrm{T}})_{[k_1,k_2]} \partial^{k_1}\partial^{k_2} \sigma_{[j_1,j_2]} \Big]+\frac{h}{2} \Big[ \partial_k b_{1,[j_1]} \sigma_1^{[k,j_2]} + \partial_k \sigma_{1,[j_1,j_2]} b_1^{[k]} \Big]\\
&+ \frac{h}{4} \Big( (\sigma_1 \sigma^{\mathrm{T}} + \sigma \sigma_1^{\mathrm{T}})_{[k_1,k_2]} \partial^{k_1} \partial^{k_2} \sigma_{1,[j_1,j_2]} + (\sigma_1 \sigma_1^{\mathrm{T}})_{[k_1,k_2]} \partial^{k_1} \partial^{k_2} \sigma_{[j_1,j_2]} \Big). 
\end{aligned}
\end{equation*}

By combining Proposition \ref{Prop: order two} with the backward error estimate \eqref{BAE err}, we obtain the following key estimate that for any test function $f\in C_P^{\infty}(\mathbb{R}^d)$, sufficiently small time step $h>0$, and discrete times $t_h = nh$, we have
    \begin{equation}\label{X'_t X_t err}
         \left|\mathbb{E}[f(\tilde{X}^h_{t_n}) - \mathbb{E}f(X'_{t_n})]\right| \leq C h^2.
    \end{equation}
 
This result justifies using the modified SDE \eqref{modified SDE} to approximate the effective diffusion matrix. Defining the modified coefficients
\begin{equation}\label{Ahp bhp}
		A_h' = \frac{1}{2}(\sigma + h^2\sigma_1')(\sigma + h^2\sigma_1')^{\mathrm{T}}, \qquad
		b_h' = b + h^2b_1',
\end{equation}
we obtain the associated generator and its adjoint operator 
\begin{align*}
    \mathcal{L}^{h} v := - A_h':\nabla^2 v  -b_h'\cdot \nabla v,\qquad    \left(\mathcal{L}^{h}\right)^* v  := - \nabla^2:\left(A_h'v\right)+ \nabla\cdot \left(b_h'v\right).
\end{align*}
When $h>0$ is small enough, we can preserve that $A_h'$ is still uniformly elliptic. By the Fredholm alternative, there exists a unique solution $r_h'\in L^2(\mathbb{T}^d)$ to
\begin{equation*}
    (\mathcal{L}^h)^*r_h' = 0\quad\text{in }\mathbb{T}^d,\qquad \int_{\mathbb{T}^d} r_h'(y)\,\mathrm{d}y = 1,
\end{equation*}
and there exists a unique solution $\chi_h'\in H^2(\mathbb{T}^d)$ to
\begin{equation*}
    \mathcal{L}^h \chi_h' = b_h' - \int_{\mathbb{T}^d} b_h'(y)r_h'(y)\, \mathrm{d}y \quad\text{in }\mathbb{T}^d,\qquad \int_{\mathbb{T}^d}\chi_h'(y)\,\mathrm{d}y = 0.
\end{equation*}
The function $r_h'$ is an approximation of the invariant measure $r$. The corresponding Lagrangian effective diffusivity $\overline{A}^L_h$ is given by
\begin{equation}\label{A^L_h}
    \overline{A}^L_h = \int_{\mathbb{T}^d}[I_d + \nabla \chi_h'(y)]A_h'(y)[I_d + \nabla \chi_h'(y)]^{\mathrm{T}} r_h'(y)\,\mathrm{d}y.
\end{equation}
A quick argument reveals an explicit stability bound for nondivergence-form problems with drift:
\begin{lemma}\label{chi_h' estimate}
    Let $A\in W^{1,\infty}(\T^d;\R^{d\times d}_{\mathrm{sym}})$ be a uniformly elliptic diffusion matrix with ellipticity constant $\lambda>0$, let $b\in L^{\infty}(\T^d;\R^d)$ be a drift vector field, and let $r\in C(\T^d;(0,\infty))$ denote the corresponding invariant measure satisfying $\mathcal{L}^*r=0$ in $\mathbb{T}^d$ and $\int_{\mathbb{T}^d} r(y)\,\mathrm{d}y = 1$. For a source term $z \in L^2(\mathbb{T}^d)$ with $\int_{\mathbb{T}^d} z(y)r(y)\,\mathrm{d}y = 0$, the weak solution $w\in H^1(\mathbb{T}^d)$ to 
    \begin{equation}\label{standard cell problem}
    	\mathcal{L} w := - A:\nabla^2 w  -b\cdot \nabla w  = z\quad\text{in }\T^d,\qquad \int_{\mathbb{T}^d}w(y)\,\mathrm{d}y = 0
    \end{equation} 
    satisfies the energy estimate 
       \begin{equation*}
    	\|w\|_{L^2(\T^d)} \leq \frac{1}{2\pi} \|\nabla w\|_{L^2(\T^d)} \le \frac{1}{4\pi^2 \lambda} \frac{\sup_{\T^d}r}{\inf_{\T^d}r} \|z\|_{L^2(\T^d)},\qquad \|\nabla^2 w\|_{L^2(\T^d)} \leq C \|z\|_{L^2(\T^d)}.
    \end{equation*}
\end{lemma}

\begin{proof}
By elliptic regularity theory, we have $w\in H^2(\T^d)$. Furthermore, the invariant measure $r$ satisfies $r\in W^{1,p}(\T^d)$ for all $p<\infty$ and $r\geq \inf_{\T^d} r =: r_0>0$ in $\T^d$; see e.g., \cite{BS17,JZ23}. Rewriting equation \eqref{standard cell problem} in divergence form after multiplication by $r$, we obtain  
\begin{align*}
-\nabla\cdot (rA\nabla w)+\beta\cdot \nabla w = rz\quad\text{in }\T^d,
\end{align*}
where $\beta = \mathrm{div}(rA)-rb$. In particular, we find that
\begin{align*}
    \int_{\mathbb{T}^d}r(y)A(y)\nabla w(y) \cdot \nabla w(y)\,\mathrm{d}y + \int_{\mathbb{T}^d}(\beta(y) \cdot \nabla w(y)) w(y)\,\mathrm{d}y = \int_{\T^d} r(y)z(y)w(y) \,\mathrm{d}y.
\end{align*}
Note that $A\geq \lambda I_n$ in $\T^d$ for some $\lambda > 0$, that $\nabla\cdot \beta = 0$ weakly in $\T^d$ by definition of $r$, and that $\int_{\T^d} r(y)z(y)\, \mathrm{d}y = 0$. Let $\hat{w}:=\int_{\T^d} w(y)\, \mathrm{d}y$ be the mean value of $w$. Using Poincar\'{e} and H\"{o}lder inequalities and the uniform lower bound on $r$, we obtain that
\begin{align*}
r_0\lambda \|\nabla w\|_{L^2(\T^d)}^2 &\leq \int_{\T^d} rA\nabla w\cdot \nabla w \,\mathrm{d}y = \int_{\T^d} rA\nabla w\cdot \nabla w \,\mathrm{d}y + \int_{\T^d} \beta\cdot\nabla \left(\frac{w^2}{2}\right) \mathrm{d}y \\ &= \int_{\T^d} rzw \,\mathrm{d}y = \int_{\T^d} rz(w-\hat{w})\, \mathrm{d}y \\
&\leq \|r\|_{L^\infty(\T^d)}\|z\|_{L^2(\T^d)}\|w-\hat{w}\|_{L^2(\T^d)} \leq \frac{1}{2\pi}\|r\|_{L^\infty(\T^d)}\|z\|_{L^2(\T^d)}\|\nabla w\|_{L^2(\T^d)},
\end{align*}
which completes the proof.
\end{proof} 

\begin{remark}\label{Rk: r equals 1}
    When $A$ is a constant symmetric positive definite matrix and $\nabla\cdot b = 0$ (weakly), then the invariant measure is $r\equiv 1$.
\end{remark}

Let $\chi' \in H^2(\mathbb{T}^d)$ and $\chi'_{h} \in H^2(\mathbb{T}^d)$ be the solutions of following cell problems corresponding to the original SDE and the modified SDE with modified equation, respectively, i.e.,
\begin{align*}
    \mathcal{L} \chi' &= b - \langle b\rangle_r\quad\text{in }\T^d,\qquad \int_{\mathbb{T}^d} \chi(y)\, \mathrm{d}y = 0,\\
    \mathcal{L}^{h}\chi_h' &= b_h' - \langle b_h' \rangle_{r'_h}\quad\text{in }\T^d,\qquad \int_{\mathbb{T}^d} \chi_h'(y)\,\mathrm{d}y = 0,
\end{align*}
where $\langle b\rangle_r := \int_{\T^d} b(y)r(y)\,\mathrm{d}y$ and $\langle b_h'\rangle_{r'_h} := \int_{\mathbb{T}^d} b_h'(y)r_h'(y)\, \mathrm{d}y$ denote the averaged drift terms with respect to the original and modified invariant measure. These solutions satisfy the relation 
\begin{align*}
    \mathcal{L} (\chi' - \chi_h') &= \mathcal{L}\chi' - \mathcal{L}^h \chi_h' + (\mathcal{L}^h-\mathcal{L})\chi_h' \\
    &= (b - b_h') + (\langle b_h' \rangle_{r_h'} -\langle b \rangle_{r}) + (A-A_h'):\nabla^2\chi_h' + (b-b_h')\cdot\nabla \chi_h'  =: S_{h}.
\end{align*}
Introducing the discrepancy term
\begin{align}\label{discr term}
    \alpha_h := \|A-A_h'\|_{L^{\infty}(\T^d)} + \|b-b_h'\|_{L^{\infty}(\T^d)},
\end{align}
and noting that, for $h>0$ sufficiently small,
\begin{align}\label{beth}
\begin{split}
    \beta_h:=\left\lvert\langle b \rangle_{r} -  \langle b_h' \rangle_{r_h'}  \right\rvert &\leq \left\lvert \langle b-b_h' \rangle_{r} \right\rvert + \left\lvert \langle b_h'-b \rangle_{r}-\langle b_h'-b \rangle_{r_h'} \right\rvert + \left\lvert \langle b \rangle_{r}-\langle b \rangle_{r_h'} \right\rvert  \\&\leq \|b-b_h'\|_{L^{\infty}(\T^d)} + \left(\|b\|_{L^{\infty}(\T^d)}+1\right)\|r-r_h'\|_{L^2(\T^d)},
    \end{split}
\end{align}
we apply Lemma \ref{chi_h' estimate} to obtain that
\begin{align*}
    \|\chi' - \chi_{h}'\|_{H^2(\mathbb{T}^d)} \le C\|S_h\|_{L^2(\mathbb{T}^d)} \leq C\left(\|b-b_h'\|_{L^{\infty}(\T^d)} + \beta_h + \alpha_h\|\chi'\|_{H^2(\T^d)} + \alpha_h \|\chi' - \chi_{h}'\|_{H^2(\mathbb{T}^d)} \right).
\end{align*} 
Since $\alpha_h = \mathcal{O}(h^2)$, for $h>0$ sufficiently small we can absorb the last term into the left-hand side to deduce the error bound
\begin{align}\label{H2 bound on chi}
    \|\chi' - \chi_{h}'\|_{H^2(\mathbb{T}^d)} \le C\left(\|b-b_h'\|_{L^{\infty}(\T^d)} + \beta_h + \alpha_h\|\chi'\|_{H^2(\T^d)} \right) = \mathcal{O}(h^2),
\end{align}
where we used $\beta_h = \mathcal{O}(h^2)$ which follows from \eqref{beth} and the following lemma (with $\alpha_h = \mathcal{O}(h^2)$).
  
\begin{lemma}\label{r_h' estimate}
    Let $A,A_h'\in W^{1,\infty}(\mathbb{T}^d,\mathbb{R}_{\mathrm{sym}}^{d\times d})$ be uniformly elliptic diffusion matrices and $b,b_h'\in L^{\infty}(\mathbb{T}^d;\mathbb{R}^d)$ be drift vector fields. 
    Consider the corresponding invariant measures $r,r_h' \in L^2(\mathbb{T}^d)$ that solve (weakly) the following problems:  
    \begin{equation}\label{rrhp pro}
        \begin{aligned}
            -\nabla^2:(Ar) + \nabla\cdot(br) &= 0\quad\text{in }\T^d,\qquad \int_{\mathbb{T}^d}r(y)\, \mathrm{d}y = 1,\\
            -\nabla^2:(A_h'r_h') + \nabla\cdot(b_h' r_h') &= 0\quad\text{in }\T^d,\qquad \int_{\mathbb{T}^d}r_h'(y)\, \mathrm{d}y= 1. 
        \end{aligned}
    \end{equation}
    Then, if the discrepency term $\alpha_h$ defined as in \eqref{discr term} satisfies $\lim_{h\rightarrow 0}\alpha_h = 0$, there exists a constant $C >0$ such that
\begin{equation}\label{rh' bd}
    \|r - r_h'\|_{L^2(\mathbb{T}^d)}  \le C\|r\|_{L^2(\mathbb{T}^d)}\left(1+\|r\|_{L^2(\mathbb{T}^d)}\right) \alpha_h
\end{equation}
for $h>0$ sufficiently small.
\end{lemma}

\begin{proof}
Let $w\in H^2(\T^d)$ be the unique solution to the elliptic problem
\begin{align*}
-A:\nabla^2 w - b\cdot \nabla w = (r - r_h') - \langle r-r_h'\rangle_r\quad\text{in }\T^d,\qquad \int_{\T^d} w(y)\, \mathrm{d}y= 0.       
\end{align*}
Multiplying the above equation by $r - r_h'$, integrating over $\mathbb{T}^d$, noting  $\int_{\T^d} (r(y)-r_h'(y))\,\mathrm{d}y = 0$, and using the (weak) solution properties of $r$ and $r_h'$ for \eqref{rrhp pro}, we obtain that 
\begin{align*}
    \|r-r_h'\|_{L^2(\T^d)}^2 &= \int_{\T^d} (r-r_h') (-A:\nabla^2 w - b\cdot \nabla w)\, \mathrm{d}y \\&= \int_{\T^d} r_h' \left((A-A_h'):\nabla^2 w + (b-b_h')\cdot \nabla w\right) \mathrm{d}y\\& \le \|r_h'\|_{L^2(\T^d)} \|w\|_{H^2(\T^d)}\alpha_h \\
    &\leq C\left(\|r-r_h'\|_{L^2(\T^d)} + \|r\|_{L^2(\T^d)}\right)\|r-r_h'\|_{L^2(\T^d)}\left(1+\|r\|_{L^2(\T^d)}\right) \alpha_h \\
    &\le C\|r-r_h'\|_{L^2(\T^d)}^2  \left(1+\|r\|_{L^2(\T^d)}\right)\alpha_h + C\|r\|_{L^2(\T^d)} \left(1+\|r\|_{L^2(\T^d)}\right)\|r-r_h'\|_{L^2(\T^d)}\alpha_h,
\end{align*}
where we used Lemma \ref{chi_h' estimate} in the penultimate step. Since $\alpha_h\rightarrow 0$ as $h\rightarrow 0$, we can absorb the first term on the right-hand side into the left-hand side and obtain the estimate \eqref{rh' bd} for $h>0$ sufficiently small.   
\end{proof}
 
Building upon the preceding estimates, we establish the following error bound for the numerical approximation of the effective diffusivity.  

\begin{theorem}\label{Thm: Abar approx}
    Let the situation be as in Theorem \ref{Relation between Lagrangian and Eulerian frameworks}, and let $A_h',b_h'$ be given by \eqref{Ahp bhp}. Then, for the effective diffusivities $\overline{A}^L$ and $\overline{A}^L_h$ defined in \eqref{Effective diffusivity} and \eqref{A^L_h} respectively, there holds 
    \begin{equation*}
        \left\lvert \overline{A}_h^L - \overline{A}^L \right\rvert  = \mathcal{O}(h^2).
    \end{equation*}
\end{theorem}

\begin{proof}
Using triangle and H\"{o}lder inequalities, we find that
\begin{align*}
    &\left\lvert \overline{A}_h^L - \overline{A}^L \right\rvert = \left\lvert\int_{\mathbb{T}^d}\left(I_d + \nabla \chi'_h\right)A'_h \left(I_d + \nabla \chi'_h\right)^{\mathrm{T}}r'_h\, \mathrm{d}y -\int_{\mathbb{T}^d}\left(I_d + \nabla \chi'\right)A\left(I_d + \nabla \chi'\right)^{\mathrm{T}}r\, \mathrm{d}y\right\rvert \\
    &\leq \left\lvert \int_{\mathbb{T}^d}\left(I_d + \nabla \chi'\right)A \left(\nabla(\chi'-\chi_h')\right)^{\mathrm{T}}r\, \mathrm{d}y \right\rvert + \left\lvert \int_{\mathbb{T}^d}\left(\nabla(\chi'-\chi_h')\right)A \left(I_d + \nabla \chi_h'\right)^{\mathrm{T}}r\, \mathrm{d}y \right\rvert \\
    &\qquad+ \left\lvert \int_{\mathbb{T}^d}\left(I_d + \nabla \chi'_h\right)\left(A-A'_h\right) \left(I_d + \nabla \chi'_h\right)^{\mathrm{T}}r'_h\, \mathrm{d}y  \right\rvert + \left\lvert \int_{\mathbb{T}^d}\left(I_d + \nabla \chi'_h\right)A \left(I_d + \nabla \chi'_h\right)^{\mathrm{T}}\left(r-r'_h\right) \mathrm{d}y  \right\rvert \\
    &\leq C\left(\|\nabla e_{\chi}\|_{L^2(\T^d)} \left(1 + \|\nabla \chi_h'\|_{L^2(\T^d)}\right) + \left( \|e_A\|_{L^{\infty}(\T^d)}\|r_h'\|_{L^2(\T^d)} + \|e_r\|_{L^2(\T^d)}\right) \left(1+\|\nabla \chi_h'\|_{L^4(\T^d)}\right)^2 \right)
\end{align*}
for some constant $C>0$, where we have written $e_{\chi}:= \chi'-\chi_h'$, $e_A := A-A_h'$, and $e_r := r-r_h'$. From \eqref{H2 bound on chi} and Lemma \ref{r_h' estimate}, we have that  
\begin{align*}
\|\nabla e_{\chi}\|_{L^2(\mathbb{T}^d)} = \mathcal{O}(h^2),\qquad \|e_A\|_{L^{\infty}(\mathbb{T}^d)} = \mathcal{O}(h^2),\qquad \|e_r\|_{L^2(\mathbb{T}^d)} = \mathcal{O}(h^2)
\end{align*}
for some constant $C>0$. Further, similarly to the derivation of \eqref{H2 bound on chi}, but using $W^{2,p}$ estimates for solutions to \eqref{standard cell problem} instead of Lemma \ref{chi_h' estimate}, we can find that $\|\nabla \chi_h'\|_{L^4(\T^d)}$ remains uniformly bounded as $h\rightarrow 0$. We conclude that
\begin{align*}
    \left\lvert \overline{A}_h^L - \overline{A}^L \right\rvert = \mathcal{O}(h^2),
\end{align*}
as required. 
\end{proof}

The following corollary reveals the connection between the effective diffusivity and the long-time covariance of the process.
\begin{corollary}
    Let the situation be as in Theorem \ref{Thm: Abar approx}. Let $\{\Tilde{X}^h_{t_n}\}_{n=1}^N$ be numerical solutions obtained from the modified Milstein scheme \eqref{Euler Maruyama with modified equation} with time step $h$, and let $\bm x_{t_N} =  (\bm x_{t_N,[1]},\bm x_{t_N,[2]},\dotsc, \bm x_{t_N,[d]})^T$ be 
    $M$ independent Monte Carlo realizations, where each $\bm x_{t_N,[i]} \in \mathbb{R}^M$ and $X_{t_N} = X_T$. Then, for $h>0$ sufficiently small, the empirical variance estimator satisfies
    \begin{equation}\label{eq:variance estimator}
        \left\lvert \frac{\mathrm{Var}(\bm x_{t_N})}{2T} - \overline{A}^L\right\rvert \leq C\left(h^2 +\frac{1}{T} +\frac{1}{\sqrt{M}}\right)
    \end{equation} 
for some constant $C>0$ independent of $h$, $T$, and $M$. 
\end{corollary}

\begin{proof}
    Using the triangle inequality, we split the error into four parts:
    \begin{equation}\label{eq:error_decomposition}
    	\begin{aligned}
    		\left\lvert \frac{\mathrm{Var}(\bm x_{t_N})}{2T} - \overline{A}^L\right\rvert &\leq \left\lvert\frac{\mathrm{Var}(\bm x_{t_N})}{2T} - \frac{\mathrm{Var}(\Tilde{X}^h_{t_N})}{2T}\right\rvert +   \left\lvert \frac{\mathrm{Var}(\Tilde{X}^h_{t_N})}{2T} -\frac{\mathrm{Var}(X'_{t_N})}{2T}\right\rvert\\ &+  \left\lvert \frac{\mathrm{Var}(X'_{t_N})}{2T} - \overline{A}_h^L\right\rvert + \left\lvert \overline{A}_h^L - \overline{A}^L\right\rvert.
    	\end{aligned}
    \end{equation}	
For the first component in \eqref{eq:error_decomposition}, by the Central Limit Theorem, there exists $L_1 >0$ such that
    \begin{equation*}
        \left\lvert \frac{\mathrm{Var}(\bm x_{t_N})}{2T} - \frac{\mathrm{Var}(\Tilde{X}^h_{t_N})}{2T}\right\rvert \leq \frac{L_1}{\sqrt{M}}.
    \end{equation*}
    Applying \eqref{X'_t X_t err} to the moment functions $f_1(x) = x$ and $f_2(x) = xx^{\mathrm{T}}$, we obtain that there exist constants $K_1, K_2 >0$ such that
        \begin{align*}
             \left\lvert \mathbb{E}\left[\tilde{X}^h_{t_N}\right] -\mathbb{E}\left[X'_{t_N}\right]\right\rvert \leq K_1h^2,\qquad
             \left\lvert \mathbb{E}\left[\left(\tilde{X}^h_{t_N}\right)\left(\tilde{X}^h_{t_N}\right)^{\mathrm{T}}\right] -\mathbb{E}\left[\left(X'_{t_N}\right)\left(X'_{t_N}\right)^{\mathrm{T}}\right]\right\rvert \leq K_2h^2
        \end{align*}
    for $h>0$ sufficiently small. From the covariance structure, we have that
   \begin{align*}
\frac{1}{2T} \left\lvert \mathrm{Var}(\tilde{X}_{t_N}^h) - \mathrm{Var}(X'_{t_N}) \right\rvert 
&\leq \frac{1}{2T} \Big\lvert \mathbb{E}\left[\left(\tilde{X}^h_{t_N}\right)\left(\tilde{X}^h_{t_N}\right)^{\mathrm{T}}\right]
- \mathbb{E}\left[\left(X'_{t_N}\right)\left(X'_{t_N}\right)^{\mathrm{T}}\right] \Big\rvert \\ &+ \frac{1}{2T} \Big\lvert \mathbb{E}\left[\tilde{X}^h_{t_N}\right] - \mathbb{E}\left[X'_{t_N}\right] \Big\rvert
\left( \Big\lvert\mathbb{E}\left[\tilde{X}^h_{t_N}\right] \Big\rvert +  \Big\lvert\mathbb{E}\left[X'_{t_N}\right]\Big\rvert\right).
\end{align*}
    Next, note that from \eqref{moment_est} in Theorem \ref{Relation between Lagrangian and Eulerian frameworks} that there exist $M_1,M_2\in \mathbb{R}^d$ and $L_2>0$ such that
\begin{equation*}
\begin{aligned}\left\lvert
\frac{\mathbb{E}\left[\tilde{X}^h_{t_N}\right]}{2T} - M_1\right\rvert + \left\lvert\frac{\mathbb{E}\left[X'_{t_N}\right]}{2T}  - M_2\right\rvert \le \frac{L_2}{T}.
\end{aligned}
\end{equation*}
Hence, using the triangle inequality, we deduce that for $h>0$ sufficiently small there holds
\begin{equation*}
\frac{1}{2T} \left\lvert \mathrm{Var}(\tilde{X}_{t_N}^h) - \mathrm{Var}(X'_{t_N}) \right\rvert \leq L_3 \left(h^2 + \frac{1}{T}\right)
\end{equation*}
for some constant $L_3>0$. Finally, \eqref{rate 1overt} in Theorem \ref{Relation between Lagrangian and Eulerian frameworks} establishes that there exists a constant $L_4>0$ such that
    \begin{equation*}
        \left\lvert \frac{\mathrm{Var}(X'_{t_N})}{2T} - \overline{A}_h^L\right\rvert \leq \frac{L_4}{T}.
    \end{equation*}
    Combining these estimates with \eqref{eq:error_decomposition} and Theorem \ref{Thm: Abar approx}, which bounds $\lvert \overline{A}_h^L - \overline{A}^L\rvert$, yields the final estimate \eqref{eq:variance estimator}.
    \end{proof}

\begin{remark}
When $b$ satisfies the centering condition \eqref{centering}, the long-time variance can be seen as the numerical approximation of $\overline{A}$ as in this case $\overline{A}^L = \overline{A}$.
\end{remark}

\section{Numerical results}
\noindent
In this section, we apply our method to study the behavior of several multidimensional diffusion processes. We compare the convergence rates of the Euler--Maruyama scheme and the Milstein scheme with modified equations, as illustrated in the following figures. 
The second-order convergence of the effective diffusion matrix by the Milstein scheme with modified equation will be confirmed by our numerical test. 

\subsection{Two-dimensional SDE with constant diffusion coefficient}
We consider the two-dimensional stochastic differential equation 
\begin{equation*}
	\left\{
	\begin{aligned}
		\mathrm{d}X^1_t &= 2\pi\sin(2\pi X^1_t)\sin(2\pi X^2_t)\,\mathrm{d}t + 2\,\mathrm{d}W^1_t,\qquad\, X_0^1 = 0,\\
		\mathrm{d}X^2_t &= 2\pi\cos(2\pi X^1_t)\cos(2\pi X^2_t)\,\mathrm{d}t + 2\,\mathrm{d}W^2_t,\qquad X_0^2 = 0,
	\end{aligned}
	\right.
\end{equation*}
which corresponds to $b(y_1,y_2) := 2\pi( \sin(2\pi y_1)\sin(2\pi y_2),\cos(2\pi y_1)\cos(2\pi y_2))^{\mathrm{T}}$ and $\sigma \equiv 2I_2$.

We note that the diffusion coefficient $A := \frac{1}{2}\sigma\sigma^{\mathrm{T}} = \sigma$ is a constant symmetric positive definite matrix, and the drift term $b$ is smooth and divergence-free. Hence, in view of Remark \ref{Rk: r equals 1}, the system admits the trivial invariant measure $r \equiv 1$.  It follows that the effective diffusivity is given by 
\begin{align}\label{Abar without r}
    \overline{A} = \int_{\mathbb{T}^d} \left(I_2 + \nabla \chi(y)\right)A(y)\left(I_2 + \nabla \chi(y)\right)^{\mathrm{T}}\,\mathrm{d}y.
\end{align}
In our numerical experiments, we plot the absolute error of $\frac{1}{2T}\mathrm{Var}(\bm{x}_{t_N,[1]})$ at the final time $T = 100$ with time steps $h \in [10^{-3},10^{-2}]$ for $10^7$ particles and $ \overline{A}_{11}$ obtained from \eqref{Abar without r} and a fine finite element approximation of $\chi$ on a fine mesh.

In Figure \ref{constant-sigma-1},  the effective diffusivity error is computed using: (i)~the standard Euler--Maruyama scheme, (ii)~the Euler--Maruyama scheme with modified equations. The results show that the Euler--Maruyama scheme with modified equations achieves second-order convergence, whereas the standard Euler--Maruyama scheme exhibits only first-order convergence, with significantly larger errors.
  
\begin{figure}[htbp]
\centering
\includegraphics[width= 10cm]{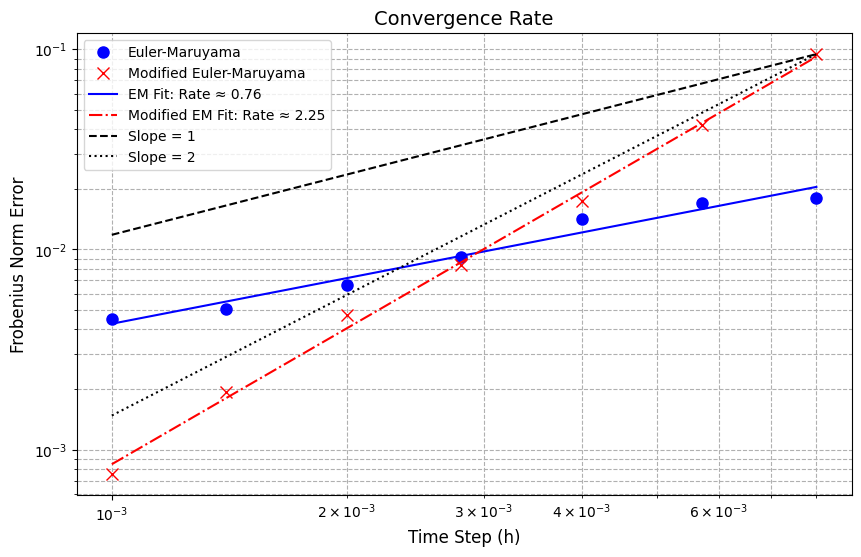}
\caption{Convergence rate of $\left|\frac{\mathrm{Var}(\bm{x}^1_n)}{2T} - \overline{A}_{11}\right|$ for the Euler--Maruyama scheme versus the modified Euler--Maruyama scheme.}
\label{constant-sigma-1}
\end{figure}

\subsection{Two-dimensional SDE with non-constant diffusion coefficient}
We now consider the two-dimensional SDE corresponding to the anisotropic diffusion coefficient 
$$A(y) :=
\begin{pmatrix}
&\frac{2 + \sin(2\pi y_1)}{2 + \sin(2\pi y_1)\sin(2\pi y_2)}  & 0\\
&0 &\frac{2 + \sin(2\pi y_2)}{2 + \sin(2\pi y_1)\sin(2\pi y_2)}
\end{pmatrix}
$$
and drift field
$$
b(y) := \left(\frac{2\pi \cos(2\pi y_1)}{2 + \sin(2\pi y_1)\sin(2\pi y_2)},\frac{2\pi\cos(2\pi y_2)}{2 + \sin(2\pi y_1)\sin(2\pi y_2)}\right)^{\mathrm{T}},
$$
for $y=(y_1,y_2)$. We observe that the system admits the invariant measure 
\begin{align*}
    r(y) := 1 + \frac{1}{2}\sin(2\pi y_1)\sin (2\pi y_2)
\end{align*}
for $y=(y_1,y_2)$, and that the drift field $b$ satisfies the centering condition.

For the numerical simulation in the Lagrangian framework, we set $ \sigma := \sqrt{2} A^{\frac{1}{2}}$, leading to the SDE system 
\begin{equation*}
\left\{
\begin{aligned}
    &\mathrm{d}X_t^1 = \frac{2\pi \cos(2\pi X_t^1)}{2 + \sin(2\pi X_t^1)\sin(2\pi X_t^2)}\,\mathrm{d}t + \sqrt{\frac{4 + 2\sin(2\pi X_t^1)}{2 + \sin(2 \pi X_t^1)\sin(2\pi X_t^2)}}\,\mathrm{d}W_t^1,\qquad X_0^1 = 0,\\
    &\mathrm{d}X_t^2 = \frac{2\pi\cos(2\pi X_t^2)}{2+\sin(2 \pi X_t^1)\sin(2\pi X_t^2)}\,\mathrm{d}t + \sqrt{\frac{4 + 2\sin(2\pi X_t^2)}{2 + \sin(2\pi X_t^1)\sin(2\pi X_t^2)}}\,\mathrm{d}W_t^2,\qquad X_0^2 = 0.
    \end{aligned}
    \right.
\end{equation*}

In Figures~\ref{fig:Non-constant_case_for}--\ref{fig:Non-constant_case_a_22}, we investigate numerical convergence for time steps $h\in [10^{-3},10^{-2}]$ with final time $T=100$, using $4\times10^6$ independent Monte Carlo trajectories. The effective diffusivity error is evaluated through two distinct approaches: (i)~the standard Euler--Maruyama scheme, and (ii)~the Milstein scheme with modified equations, where a fine finite element approximation serves as a reference. The results demonstrate that the modified Milstein scheme achieves significantly higher accuracy compared to the standard Euler--Maruyama scheme.

The convergence results are presented in Figures~\ref{fig:Non-constant_case_for}--\ref{fig:Non-constant_case_a_22}. Figure~\ref{fig:Non-constant_case_for} shows the Frobenius norm error of the effective diffusivity tensor as a function of the time step size $h$, where we denote $\{\bm x_{t_N}^i = (x_{t_N,[1]}^i,x_{t_N,[2]}^i,\dotsc, x_{t_N,[d]}^i)\}_{i=1}^M$. The standard Euler--Maruyama scheme exhibits a convergence rate close to first order, which is consistent with its theoretical weak convergence properties. In contrast, the Milstein scheme with modified equations achieves a higher convergence rate, closely approaching second-order convergence.

Figures~\ref{fig:Non-constant_case_a_11} and~\ref{fig:Non-constant_case_a_22} further show the convergence behavior for the individual components $A_{11}$ and $A_{22}$ of the effective diffusivity tensor. The modified scheme consistently outperforms the standard Euler--Maruyama scheme, not only reducing the error magnitude but also significantly improving the order of convergence.

\begin{figure}[htbp]
\centering
\includegraphics[width=10cm]{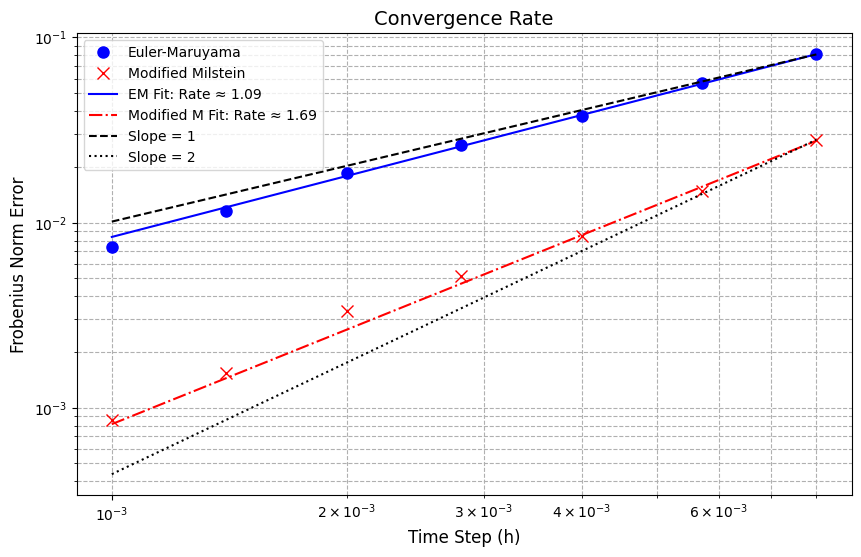}

\caption{Convergence rate  of Frobenius norm via Euler--Maruyama and modified Milstein}
\label{fig:Non-constant_case_for}
\end{figure}

\begin{figure}
\begin{minipage}[t]{0.48\textwidth}
\centering
\includegraphics[width=7cm]{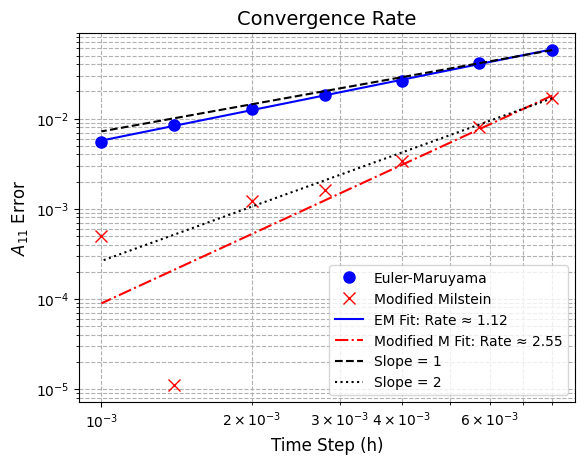}
\caption{Convergence rate of $A_{11}$ of Euler--Maruyama and modified Milstein}
\label{fig:Non-constant_case_a_11}
\end{minipage}
\begin{minipage}[t]{0.48\textwidth}
\centering
\includegraphics[width=7cm]{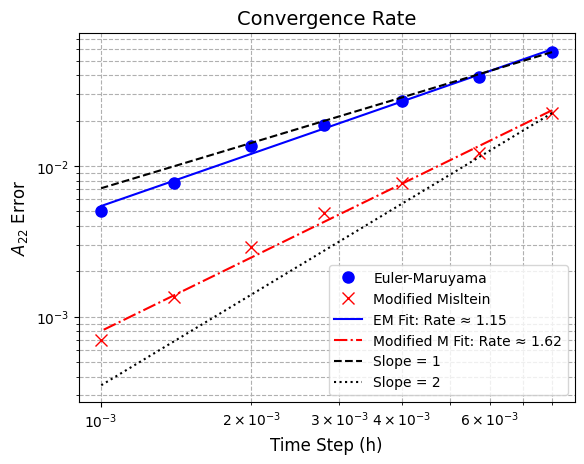}
\caption{Convergence rate of $A_{22}$ of Euler--Maruyama and modified Milstein}
\label{fig:Non-constant_case_a_22}
\end{minipage}
\end{figure}

\subsection{Three-dimensional SDE with non-constant diffusion coefficient}
We consider a three-dimensional SDE system characterized by the diffusion tensor 
$$A(y) :=
\begin{pmatrix}
&\frac{2 + \sin(2\pi y_1)}{2 + \sin(2\pi y_2)\sin(2\pi y_3)}  & 0 &0\\
&0 &\frac{2 + \sin(2\pi y_2)}{2 + \sin(2\pi y_2)\sin(2\pi y_3)} &0\\
&0 &0 &\frac{2 + \sin(2\pi y_3)}{2 + \sin(2\pi y_2)\sin(2\pi y_3)}
\end{pmatrix}
$$
coupled with the drift vector field
$$
b(y) := \left(\frac{2\pi \cos(2\pi y_1)}{ 2 + \sin(2\pi y_2)\sin(2\pi y_3)},\frac{2\pi \cos(2\pi y_2)}{ 2+\sin(2\pi y_2)\sin(2\pi y_3)},\frac{2\pi \cos(2\pi y_3)}{ 2+\sin(2\pi y_2)\sin(2\pi y_3)}\right)^{\mathrm{T}}
$$
for $y=(y_1,y_2,y_3)$. It is quickly seen that the system possesses the invariant measure 
\begin{align*}
    r(y) = 1 + \frac{1}{2}\sin (2\pi y_2)\sin(2\pi y_3)
\end{align*}
for $y=(y_1,y_2,y_3)$, and that the centering condition is satisfied.
 
 In our numerical experiments, we choose time step sizes $h$ ranging from $10^{-3}$ to $10^{-2}$, and set the final time $T = 20$. We generate $4\times10^6$ independent Monte Carlo particles. The reference effective diffusivity is determined using the result at $\Delta t = 5\times10^{-4}$. 
 
 Figures~\ref{fig:3d_case_for}--\ref{fig:3d_case_a_22} demonstrate the convergence results of the standard Euler--Maruyama scheme and the modified Milstein scheme for computing the effective diffusivity in the three-dimensional SDE with non-constant diffusion. Figure~\ref{fig:3d_case_for}
 plots the Frobenius norm error of the effective diffusivity tensor against the time step size $h$ on a log-log scale. The standard Euler--Maruyama scheme exhibits a first-order convergence rate, as expected from its weak order of accuracy, while the modified scheme achieves a significantly higher convergence rate, closely approaching second order.

The improved accuracy of the modified scheme is further evident in the component-wise errors, as shown in Figures~\ref{fig:3d_case_a_11} and~\ref{fig:3d_case_a_22} for $A_{11}$ and $A_{33}$, respectively. For these tensor components, the modified Milstein scheme consistently yields both lower error magnitudes and higher convergence rates, with observed rates exceeding two in some cases. This consistent improvement demonstrates the effectiveness of the modification in capturing the complex dynamics induced by the spatially varying, anisotropic diffusion tensor.

In addition, Figure \ref{fig:heat map of exact density and numerical density} compares the empirical steady-state density obtained from long-time simulations with the theoretical invariant measure. The close agreement between the two distributions confirms that the numerical scheme preserves the qualitative features of the invariant distribution. The second-order weak convergence of the invariant measure by the modified scheme can be directly derived from Corollary \ref{Convergence of invariant measure}. The relevant numerical experiment has been presented in \cite{ACV12}. 

These numerical results demonstrate substantial benefits of the modified Milstein scheme for simulating three-dimensional SDEs with non-constant diffusion coefficients. Second-order convergence and improved accuracy are especially valuable in practical computations, where reducing the discretization error without using very small time steps can lead to significant computational savings.

\begin{figure}[htbp]
\centering
\includegraphics[width=10cm]{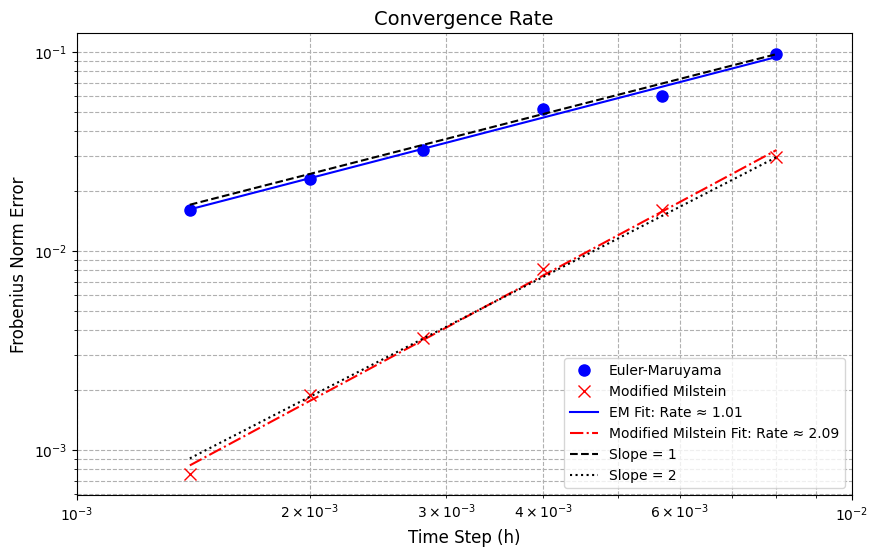}
\caption{Convergence rate  of Frobenius norm via Euler--Maruyama and modified Milstein}
\label{fig:3d_case_for}
\end{figure}
\begin{figure}
\begin{minipage}[t]{0.48\textwidth}
\centering
\includegraphics[width=7cm]{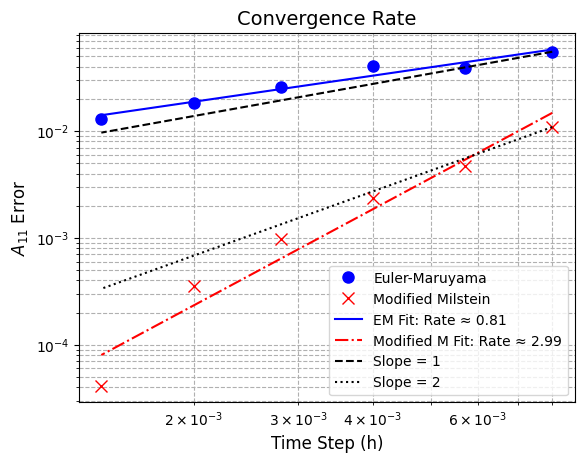}
\caption{Convergence rate of $A_{11}$ of Euler--Maruyama and modified Milstein}
\label{fig:3d_case_a_11}
\end{minipage}
\begin{minipage}[t]{0.48\textwidth}
\centering
\includegraphics[width=7cm]{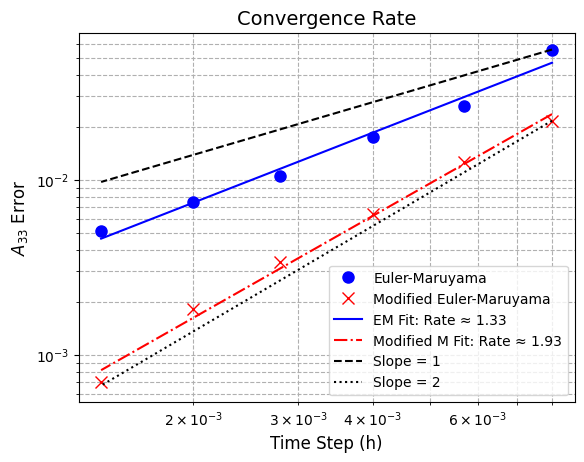}
\caption{Convergence rate of $A_{33}$ of Euler--Maruyama and modified Milstein}
\label{fig:3d_case_a_22}
\end{minipage}
\end{figure}

\begin{figure}
    \begin{minipage}[t]{0.48\textwidth}
    \centering
    \includegraphics[width= 8cm]{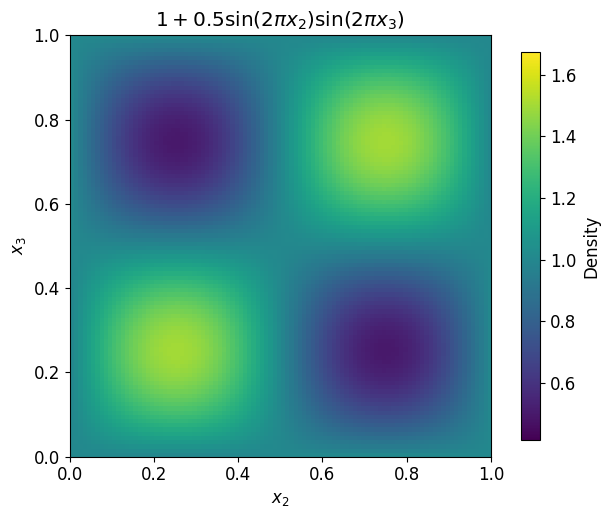}
    \label{fig:3d_case_density_exact}
    \end{minipage}
    \begin{minipage}[t]{0.48\textwidth}
    \includegraphics[width= 8cm]{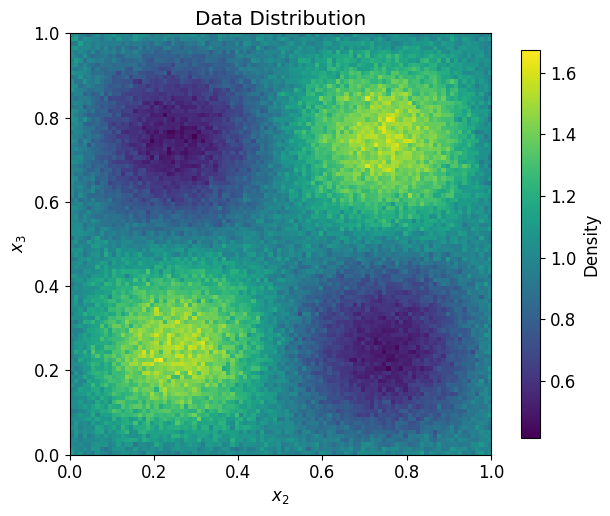}
    \label{fig:3d_case_density_num}
    \end{minipage}
    \caption{heat map of exact density and numerical density for the invariant measure}
    \label{fig:heat map of exact density and numerical density}
\end{figure}
\section{Conclusion}
\noindent
In this paper, we have developed a probabilistic numerical framework for accurately computing the effective diffusion matrix in periodic homogenization of nondivergence-form elliptic equations with large drift. Our rigorous analysis bridges the Lagrangian and Eulerian perspectives for effective diffusivity computation, extending previous results to scenarios where the centering condition fails.
This helps explain previous numerical difficulties observed with certain schemes. We introduced a modified equation approach that achieves second-order weak convergence, improving upon standard first-order methods. We derived complete error estimates for the approximation of the effective diffusivity. 

Numerical experiments in 2D and 3D demonstrated the superior convergence properties of our modified scheme compared to standard methods, achieving second-order accuracy even for non-constant diffusion coefficients. The Lagrangian particle method proves particularly effective for studying convection-enhanced diffusion in time-dependent chaotic flows. Our framework offers an efficient, accurate approach for homogenization problems with large drift, with potential applications to turbulent transport in fluid dynamics and anomalous diffusion in biological systems. 

\section*{Acknowledgement}
\noindent

ZZ was supported by the National Natural Science Foundation of China (Projects 12171406 and 92470103), the Hong Kong RGC grant (Projects 17307921 and 17304324), the Outstanding Young Researcher Award of HKU (2020-21), and Seed Funding for Strategic Interdisciplinary Research Scheme 2021/22 (HKU). The computations were performed using research computing facilities offered by Information Technology Services at the University of Hong Kong.

\bibliographystyle{plain}
\bibliography{ref_SWZ}
\end{document}